\documentclass[10pt,leqno]{amsart}
\usepackage{amsmath}
\usepackage{amsfonts}
\usepackage{amssymb}
\usepackage{graphicx}
\usepackage{color}
\usepackage{hyperref}
\usepackage{xcolor}
\theoremstyle{plain}
\newtheorem{theorem}{Theorem}[section]
\newtheorem{proposition}{Proposition}[section]
\newtheorem{corollary}[theorem]{Corollary}
\newtheorem{lemma}{Lemma}[section]

\theoremstyle{definition} 
\newtheorem{definition}[theorem]{Definition}
\newtheorem*{definition*}{Definition}

\newtheorem{question}[theorem]{Question}

\newcommand{\R}{\mathbb{R}}

\def \b {\beta}

\def\Ric{\text{Ric}}

\def\a{\alpha}

\def\g{\gamma}

\def\C{\mathbb{C}}
\def\R{\mathbb{R}}
\def\Z{\mathbb{Z}}

\def\S{\mathbb{S}}

\def\CP{\mathbb{CP}}
\def\CH{\mathbb{CH}}

\def\vp{\varphi}

\def\id{\operatorname{id}}
\def\Ric{\operatorname{Ric}}

\def\tr{\operatorname{tr}}

\def\spn{\operatorname{span}}

\numberwithin{equation}{section}
\makeatletter
\newcommand*\owedge{\mathpalette\@owedge\relax}
\newcommand*\@owedge[1]{%
  \mathbin{%
    \ooalign{%
      $#1\m@th\bigcirc$\cr
      \hidewidth$#1\m@th\wedge$\hidewidth\cr
    }%
  }%
}
\makeatother

\begin{document}

\title[K\"ahler manifolds and secondary curvature operator]{K\"ahler manifolds and the curvature operator of the second kind}

\author{Xiaolong Li}\thanks{The author's research is partially supported by Simons Collaboration Grant \#962228 and a start-up grant at Wichita State University}
\address{Department of Mathematics, Statistics and Physics, Wichita State University, Wichita, KS, 67260}
\email{xiaolong.li@wichita.edu}



\begin{abstract}
This article aims to investigate the curvature operator of the second kind on K\"ahler manifolds.  The first result states that an $m$-dimensional K\"ahler manifold with $\frac{3}{2}(m^2-1)$-nonnegative (respectively, $\frac{3}{2}(m^2-1)$-nonpositive) curvature operator of the second kind must have constant nonnegative (respectively, nonpositive) holomorphic sectional curvature. The second result asserts that a closed $m$-dimensional K\"ahler manifold with $\left(\frac{3m^3-m+2}{2m}\right)$-positive curvature operator of the second kind has positive orthogonal bisectional curvature, thus being biholomorphic to $\CP^m$. We also prove that $\left(\frac{3m^3+2m^2-3m-2}{2m}\right)$-positive curvature operator of the second kind implies positive orthogonal Ricci curvature. Our approach is pointwise and algebraic. 
\end{abstract}

\maketitle

\section{Introduction}
The Riemann curvature tensor on a Riemannian manifold $(M^n, g)$ induces a self-adjoint operator $\overline{R}:S^2(T_pM) \to S^2(T_pM)$ via 
\begin{equation*}
    \overline{R}(\vp)_{ij}=\sum_{k,l=1}^n R_{iklj}\vp_{kl},
\end{equation*}
where $S^2(T_pM)$ is the space of symmetric two-tensors on the tangent space $T_pM$. 
\textit{The curvature operator of the second kind}, denoted by $\mathring{R}$ throughout this article, refers to the symmetric bilinear form
$$\mathring{R}:S^2_0(T_pM)\times S^2_0(T_pM) \to \R$$ 
obtained by restricting $\overline{R}$ to $S^2_0(T_pM)$, the space of traceless symmetric two-tensors. See \cite{CGT21} or \cite{Li21} for a detailed discussion.
This terminology is due to Nishikawa \cite{Nishikawa86}, who conjectured in 1986 that a closed Riemannian manifold with positive (respectively, nonnegative) curvature operator of the second kind is diffeomorphic to a spherical space form (respectively,  Riemannian locally symmetric space).

Nishikawa's conjecture had remained open for more than three decades before its positive part was resolved by Cao, Gursky, and Tran \cite{CGT21} recently and its nonnegative part was settled by the author \cite{Li21} shortly after. 
The key observation in \cite{CGT21} is that two-positive curvature operator of the second kind implies the strictly PIC1 condition (i.e., $M\times \R$ has positive isotropic curvature). 
This is sufficient since earlier work of Brendle \cite{Brendle08} has shown that a solution to the normalized Ricci flow starting from a strictly PIC1 metric on a closed manifold exists for all time and converges to a metric of constant
positive sectional curvature. 
Soon after that, the author \cite{Li21} proved that strictly PIC1 is implied by three-positivity of $\mathring{R}$, thus getting an immediate improvement to the result in \cite{CGT21}. 
Furthermore, the author was able to resolve the nonnegative case of Nishikawa's conjecture under three-nonnegativity of $\mathring{R}$. 
More recently, the conclusion has been strengthened by Nienhaus, Petersen, and Wink \cite{NPW22}, who ruled out irreducible compact symmetric spaces by proving that an $n$-dimensional closed Riemannian manifold with $\frac{n+2}{2}$-nonnegative $\mathring{R}$ is either flat or a rational homology sphere. 
Combining these works together, we have
\begin{theorem}\label{thm 3 positive}
Let $(M^n,g)$ be a closed Riemannian manifold of dimension $n\geq 3$. 
\begin{enumerate}
    \item If $M$ has three-positive curvature operator of the second kind, then $M$ is diffeomorphic to a spherical space form. 
    \item If $M$ has three-nonnegative curvature operator of the second kind, then $M$ is either flat or diffeomorphic to a spherical space form. 
\end{enumerate}
\end{theorem}

To prove sharper results, the author \cite{Li22JGA} introduced the notion of $\alpha$-positive curvature operator of the second kind for $\alpha \in  [1,N]$, where $N=\dim(S^2_0(T_p M))$. Hereafter, $\lfloor x \rfloor$ denotes the floor function defined by
\begin{equation*}
    \lfloor x \rfloor := \max \{k \in \Z: k \leq x\}. 
\end{equation*}
\begin{definition*}
A Riemannian manifold $(M^n,g)$ is said to have $\alpha$-positive (respectively, $\a$-nonnegative) curvature operator of the second kind if for any $p\in M$ and any orthonormal basis $\{\vp_i\}_{i=1}^N$ of $S^2_0(T_pM)$, it holds that
\begin{equation}\label{alpha positive def}
     \sum_{i=1}^{\lfloor \a \rfloor} \mathring{R}(\vp_i,\vp_i) +(\a -\lfloor \a \rfloor) \mathring{R}(\vp_{\lfloor \a \rfloor+1},\vp_{\lfloor \a \rfloor+1}) > (\text{respectively,} \geq) \  0. 
\end{equation}
Similarly, $(M^n,g)$ is said to have $\a$-negative (respectively, $\a$-nonpositive) curvature operator of the second kind if the reversed inequality holds. 
\end{definition*}

In \cite{Li22JGA}, the author proved that $\left(n+\frac{n-2}{n}\right)$-positive (respectively, $\left(n+\frac{n-2}{n}\right)$-nonnegative) curvature operator of the second kind implies positive (respectively, nonnegative) Ricci curvature in all dimensions. Combined with Hamilton's work \cite{Hamilton82, Hamilton86}, this immediately leads to an improvement of Theorem \ref{thm 3 positive} in dimension three: a closed three-manifold with $3\frac{1}{3}$-positive $\mathring{R}$ is diffeomorphic to a spherical space form and a closed three-manifold with $3\frac{1}{3}$-nonnegative $\mathring{R}$ is either flat, or diffeomorphic to a spherical space form, or diffeomorphic to a quotient of $\S^2 \times \R$.
Note that the number $3\frac{1}{3}$ is optimal, as the eigenvalues of $\mathring{R}$ on $\S^2 \times \S^1$ are given by $\{-\frac{1}{3}, 0, 0, 1, 1\}$ (see \cite[Example 2.6]{Li21}). 

Another result obtained in \cite{Li22JGA} states that $4\frac{1}{2}$-positive (respectively, $4\frac{1}{2}$-nonnegative) curvature operator of the second kind implies positive (respectively, nonnegative) isotropic curvature in dimensions four and above. This was previously shown in \cite{CGT21} under the stronger condition of four-positivity (respectively, four-nonnegativity) of $\mathring{R}$. 
In view of Micallef and Moore's work \cite{MM88}, one concludes that a closed Riemannian manifold of dimension $n\geq 4$ with $4\frac{1}{2}$-positive $\mathring{R}$ is homeomorphic to a spherical space form. Moreover, the ``homeomorphism" can be upgraded to ``diffeomorphism" if either $n=4$ or $n\geq 12$, using Hamilton's work \cite{Hamilton97} or that of Brendle \cite{Brendle19}, respectively. 
The number $4\frac{1}{2}$ is optimal in dimension four as both $\CP^2$ and $\S^3\times \S^1$ has $\a$-positive $\mathring{R}$ for any $\a>4\frac{1}{2}$. 
In addition, a rigidity result for closed manifolds with $4\frac{1}{2}$-nonnegative $\mathring{R}$ was also obtained in \cite[Theorem 1.4]{Li22JGA} (see also \cite[Theorem 1.4]{Li22product} for an improvement).

This article aims to investigate the curvature operator of the second kind on K\"ahler manifolds. Such study dates at least back to Bourguignon and Karcher \cite{BK78}, who computed in 1978 that the eigenvalues of $\mathring{R}$ on $(\CP^m,g_{FS})$, the complex projective space with the Fubini-Study metric normalized with constant holomorphic sectional curvatures $4$, are given by: $-2$ with multiplicity $m^2-1$ and $4$ with multiplicity $m(m+1)$. 
One immediately sees that $\CP^m$ has $\a$-positive $\mathring{R}$ for any $\a > \frac{3}{2}(m^2-1)$ but not for any $\a \leq \frac{3}{2}(m^2-1)$. 
As $\CP^m$ is often considered as the K\"ahler manifold with ``most positive curvature", one may speculate that there might not be any K\"ahler manifold with $\a$-positive $\mathring{R}$ for $\a \leq \frac{3}{2}(m^2-1)$.
Our first main result confirms this speculation. 
\begin{theorem}\label{thm flat}
Let $(M^m,g,J)$ be a K\"ahler manifold of complex dimension $m\geq 2$. 
\begin{enumerate}
    \item If $M$ has $\a$-nonnegative (respectively, $\a$-nonpositive) curvature operator of the second kind for some $\a < \frac{3}{2}(m^2-1)$, then $M$ is flat. 
    \item If $M$ has $\frac{3}{2}(m^2-1)$-nonnegative (respectively, $\frac{3}{2}(m^2-1)$-nonpositive) curvature operator of the second kind, then $M$ has constant nonnegative (respectively, nonpositive) holomorphic sectional curvature.   
\end{enumerate}
\end{theorem}

Previously, the author \cite[Theorem 1.9]{Li21} proved that K\"ahler manifolds with four-nonnegative $\mathring{R}$ are flat, and Nienhaus, Petersen, Wink, and Wylie \cite{NPWW22} proved the following result.
\begin{theorem}\label{thm NPWW}
Let $(M^m,g,J)$ be a K\"ahler manifold of complex dimension $m \geq 2$. Set 
\begin{equation*}
    A=\begin{cases} 3m\frac{m+1}{m+2}, & \text{ if } $m$ \text{ is even}, \\
    3m\frac{(m+1)(m^2-1)}{(m+2)(m^2+1)}, & \text{ if } $m$ \text{ is odd}.
    \end{cases}
\end{equation*}
If the curvature operator of the second kind of $M$ is $\a$-nonnegative or $\a$-nonpositive for some $\a <A$, then $M$ is flat.
\end{theorem}

Theorem \ref{thm flat} improves Theorem \ref{thm NPWW}. Moreover, the number $\frac{3}{2}(m^2-1)$ is sharp as $(\CP^m,g_{FS})$ has $\frac{3}{2}(m^2-1)$-nonnegative $\mathring{R}$ and $(\CH^m, g_{\text{stand}})$, the complex hyperbolic space with constant negative holomorphic sectional curvature, has $\frac{3}{2}(m^2-1)$-nonpositive $\mathring{R}$. Theorem \ref{thm flat} has two immediate corollaries. 
\begin{corollary}\label{cor 1}
For $\a \leq \frac{3}{2}(m^2-1)$, there do not exist $m$-dimensional K\"ahler manifolds with $\a$-positive or $\a$-negative curvature operator of the second kind.
\end{corollary}

\begin{corollary}\label{cor 3}
Let $(M^m,g,J)$ be a complete non-flat K\"ahler manifold of complex dimension $m$. 
If $M$ has $\frac{3}{2}(m^2-1)$-nonnegative (respectively, $\frac{3}{2}(m^2-1)$-nonpositive) curvature operator of the second kind, then $M$ is isometric to $(\CP^m, g_{FS})$ (respectively, a quotient of $(\CH^m, g_{\text{stand}})$), up to scaling.
\end{corollary}

In K\"ahler geometry, there are several positivity conditions on curvatures that characterize $\mathbb{CP}^m$ among closed K\"ahler manifolds up to biholomorphism. 
For instance, a closed K\"ahler manifold with positive bisectional curvature is biholomorphic to $\mathbb{CP}^m$. This was known as the Frankel conjecture \cite{Frankel61} and it was proved independently by Mori \cite{Mori79} and Siu and Yau \cite{SY80}. 
A weaker condition, called positive orthogonal bisectional curvature, also characterizes $\mathbb{CP}^m$. This is due to Chen \cite{Chen07} and Gu and Zhang \cite{GZ10} (see also Wilking \cite{Wilking13} for an alternative proof using Ricci flow).
Therefore, it is natural to seek a positivity condition on $\mathring{R}$ that characterizes $\mathbb{CP}^m$. Regarding this question, we prove that 
\begin{theorem}\label{thm OB +}
    A closed K\"ahler manifold of complex dimension $m\geq 2$ with $\a_m$-positive curvature operator of the second kind, where
     \begin{equation}\label{eq alpha m def}
     \a_m:=  \frac{3m^3-m+2}{2m},
    \end{equation}
    is biholomorphic to $\CP^m$.
\end{theorem} 

Note that when $m=2$, $\a_2=6$ is the best constant for Theorem \ref{thm OB +} to hold, as $\CP^1 \times \CP^1$ has $(6+\epsilon)$-positive $\mathring{R}$ for $\epsilon >0$.
For K\"ahler surfaces, the author proved in \cite{Li22PAMS} that a closed K\"ahler surface with six-positive $\mathring{R}$ is biholomorphic to $\mathbb{CP}^2$, and a closed nonflat K\"ahler surface with six-nonnegative $\mathring{R}$ is either biholomorphic to $\mathbb{CP}^2$ or isometric to $\CP^1 \times \CP^1$, up to scaling. The two different proofs in \cite{Li22PAMS} only work for complex dimension two.

The number $\a_m$, however, does not seem to be optimal for $m\geq 3$. An ambitious question is 
\begin{question}
What is the largest number $C_m$ so that a closed $m$-dimensional K\"ahler manifold with $C_m$-positive curvature operator of the second kind is biholomorphic to $\mathbb{CP}^m$?   
\end{question}

Theorem \ref{thm OB +} implies $\a_m \leq C_m$. We point out that 
\begin{equation}\label{eq beta m def}
    C_m \leq \b_m:=\frac{3m^3+2m^2-3m-2}{2m},
\end{equation}
as the product manifold $\CP^{m-1}\times \CP^1$ has $\a$-positive $\mathring{R}$ for any $\a > \b_m$ (see \cite{Li22product}). 
This determines the leading term of $C_m$ to be $\frac{3}{2}m^2$. 
In particular, we have
\begin{equation*}
    \frac{40}{3} \leq C_3 \leq \frac{44}{3}. 
\end{equation*}
It remains an interesting question to determine $C_m$ for $m\geq 3$.  


In the next result, we show that $\b_m$-positivity of $\mathring{R}$ implies positive orthogonal Ricci curvature. 
\begin{theorem}\label{thm Ric perp}
Let $(M^m,g,J)$ be a K\"ahler manifold of complex dimension $m\geq 2$. 
If $M$ has $\b_m$-positive curvature operator of the second kind with $\b_m$ defined in \eqref{eq beta m def}, then $M$ has positive orthogonal Ricci curvature, namely 
\begin{equation*}
    \Ric(X,X)-R(X,JX,X,JX)/|X|^2 >0
\end{equation*}
for any $0 \neq X\in T_pM$ and any $p\in M$. 
If $M$ is further assumed to be closed, then $M$ has $h^{p,0}=0$ for any $1\leq p \leq m$, and in particular, $M$ is simply-connected and projective. 
\end{theorem}

The orthogonal Ricci curvature $\Ric^\perp$ is defined as 
\begin{equation*}
    \Ric^\perp(X,X)=\Ric(X,X)-R(X,JX,X,JX)/|X|^2
\end{equation*}
for $0 \neq X\in T_pM$. This notion of curvature 
was introduced by Ni and Zheng \cite{NZ18} in the study of Laplace comparison theorems on K\"ahler manifolds. We refer the reader to \cite{NZ19}, \cite{NWZ21}, and \cite{Ni21} for a more detailed account of it. 
The constant $\b_m$ in Theorem \ref{thm Ric perp} is optimal, as $\CP^{m-1}\times \CP^1$ has $\b_m$-nonnegative $\mathring{R}$ and it has nonnegative (but not positive) orthogonal Ricci curvature.

In addition, we prove a result similar to Theorem \ref{thm Ric perp}, which states that
\begin{theorem}\label{thm Ricci}
Let $(M^m,g,J)$ be a K\"ahler manifold of complex dimension $m\geq 2$. 
Suppose $M$ has $\g_m$-positive curvature operator of the second kind, where
\begin{equation}\label{eq gamma m def}
    \g_m:= \frac{3m^2+2m-1}{2}.
\end{equation}
Then for any $p\in M$ and $0 \neq X\in T_pM$, it holds that
\begin{equation}\label{eq mixed curvature}
    2\Ric(X,X)-R(X,JX,X,JX)/|X|^2 >0.
\end{equation}
If $M$ is further assumed to be closed, then $M$ has $h^{p,0}=0$ for any $1\leq p \leq m$, and in particular, $M$ is simply-connected and projective. 
\end{theorem}

Chu, Lee, and Tam \cite{CLT20} introduced a family of curvature conditions for K\"ahler manifolds called mixed curvature. They are defined as
$$\mathcal{C}_{a,b}(X):=a \Ric(X,\bar{X})+b R(X,\bar{X},X,\bar{X})/|X|^2$$ for $a,b\in \R$. Theorem \ref{thm Ricci} establishes a connection between $\mathring{R}$ and the mixed curvature condition $\mathcal{C}_{2,-1}$.

Finally, let's discuss our strategies to prove the above-mentioned results. We will work pointwise and establish relationships between the curvature operator of the second kind and other frequently used curvature notions in K\"ahler geometry, such as holomorphic sectional curvature, orthogonal bisectional curvature, and orthogonal Ricci curvature. 
Theorems \ref{thm flat}, \ref{thm OB +}, \ref{thm Ric perp}, and \ref{thm Ricci} follow immediately from parts (1), (2), (3), and (4) of the following theorem, respectively.
\begin{theorem}\label{thm algebra R}
    Let $(V,g,J)$ be a complex Euclidean vector space with complex dimension $m\geq 2$ and $R$ be a K\"ahler algebraic curvature operator on $V$ (see Definition \ref{def Kahler algebraic curvature operator}). Then the following statements hold:
    \begin{enumerate}
    \item If $R$ has $\frac{3}{2}(m^2-1)$-nonnegative (respectively, $\frac{3}{2}(m^2-1)$-nonpositive) curvature operator of the second kind, then $R$ has constant nonnegative (respectively, nonpositive) holomorphic sectional curvature. 
    \item If $R$ has $\a_m$-nonnegative (respectively, $\a_m$-positive, $\a_m$-nonpositive, $\a_m$-negative) curvature operator of the second kind with $\a_m$ defined in \eqref{eq alpha m def}, then $R$ has nonnegative (respectively, positive, nonpositive, negative) orthogonal bisectional curvature and nonnegative (respectively, positive, nonpositive, negative) holomorphic sectional curvature. 
    \item If $R$ has $\b_m$-nonnegative (respectively, $\b_m$-positive, $\b_m$-nonpositive, $\b_m$-negative) curvature operator of the second kind with $\b_m$ defined in \eqref{eq beta m def}, then $R$ has nonnegative (respectively, positive, nonpositive, negative) orthogonal Ricci curvature.
   \item If $R$ has $\g_m$-nonnegative (respectively, $\g_m$-positive, $\g_m$-nonpositive, $\g_m$-negative) curvature operator of the second kind with $\g_m$ defined in \eqref{eq gamma m def}, then the expression 
   \begin{equation*}
     2\Ric(X,X)-R(X,JX,X,JX)/|X|^2
  \end{equation*}
   is nonnegative (respectively, positive, nonpositive, negative) for any $0 \neq X \in V$. 
    \end{enumerate}
\end{theorem}

The strategy to prove a statement in Theorem \ref{thm algebra R} is to choose a model space and apply $\mathring{R}$ to the eigenvectors of the curvature operator of the second kind on this model space. A good model space leads to a sharp result.
The model spaces we use for parts (1), (3), and (4) of Theorem \ref{thm algebra R} are $\CP^m$, $\CP^{m-1}\times \CP^1$, and $\CP^{m-1}\times \C$, respectively. 
For part (2) of Theorem \ref{thm algebra R}, we use $\CP^m$ as the model space, but the result does not seem to be sharp for $m\geq 3$. 
Finally, it's worth mentioning that this strategy has been used successfully by the author in several works \cite{Li21, Li22JGA, Li22PAMS, Li22product} with $\CP^2$, $\CP^1 \times \CP^1$, $\S^{n-1}\times \S^1$, $\S^{k}\times \S^{n-k}$ and $\CP^k \times \CP^{m-k}$ as model spaces. 

We emphasize that our approach is pointwise; therefore, many of our results are of pointwise nature and the completeness of the metric is not needed. Another feature is that our proofs are purely algebraic and work equally well for nonpositivity conditions on $\mathring{R}$.

This article is organized as follows. 
Section 2 consists of three subsections. We fix some notation and conventions in subsection 2.1 and give an introduction to the curvature operator of the second kind in subsection 2.2. In subsection 2.3, we review some basics about K\"ahler algebraic curvature operators. 
In Section 3, we collect some identities that will be frequently used in this paper. 
In Section 4, we construct an orthonormal basis of the space of traceless symmetric two-tensors on a complex Euclidean vector space and calculate the diagonal elements of the matrix representing $\mathring{R}$ with respect to this basis.  
The proofs of Theorems \ref{thm flat}, \ref{thm OB +}, \ref{thm Ric perp}, and \ref{thm Ricci} are given in Sections 5, 6, 7, and 8, respectively.

\section{Preliminaries}
\subsection{Notation and Conventions}

In the following, $(V,g)$ is a real Euclidean vector space of dimension $n \geq 2$ and $\{e_i\}_{i=1}^n$ is an orthonormal basis of $V$. We always identify $V$ with its dual space $V^*$ via the metric. 
\begin{itemize}
    \item $S^2(V)$ and $\Lambda^2(V)$ denote the space of symmetric two-tensors on $V$ and two-forms on $V$, respectively. 
    \item $S^2_0(V)$ denotes the space of traceless symmetric two-tensors on $V$. Note that $S^2(V)$ splits into $O(V)$-irreducible subspaces as
    \begin{equation*}
        S^2(V)=S^2_0(V)\oplus \R g. 
    \end{equation*}
    \item $S^2(\Lambda^2 V)$, the space of symmetric two-tensors on $\Lambda^2(V)$, has the orthogonal decomposition 
    \begin{equation*}
        S^2(\Lambda^2 V) =S^2_B(\Lambda^2 V) \oplus \Lambda^4 V,
    \end{equation*}
    where $S^2_B(\Lambda^2 V)$ consists of all tensors $R\in S^2(\Lambda^2(V))$ that also satisfy the first Bianchi identity. The space $S^2_B(\Lambda^2(V))$ is called the space of algebraic curvature operators (or tensors) on $V$.
    \item The tensor product is defined via
    \begin{equation*}
        (e_i\otimes e_j)(e_k,e_l)=\delta_{ik}\delta_{jl}. 
    \end{equation*}
    \item $\odot$ denotes the symmetric product defined by 
    \begin{equation*}
        u \odot v=u\otimes v +v \otimes u. 
    \end{equation*}
    \item $\wedge$ denotes the wedge product defined by 
    \begin{equation*}
        u \wedge v=u\otimes v - v \otimes u. 
    \end{equation*}
    \item The inner product on $S^2(V)$ is given by 
    \begin{equation*}
        \langle A, B \rangle =\tr(A^T B). 
    \end{equation*}
    If $\{e_i\}_{i=1}^n$ is an orthonormal basis of $V$, then $\{\frac{1}{\sqrt{2}}e_i \odot e_j\}_{1\leq i<j\leq n} \cup \{\frac{1}{2}e_i \odot e_i\}_{1\leq i\leq n}$ is an orthonormal basis of $S^2(V)$. 
    \item The inner product on $\Lambda^2(V)$ is given by 
    \begin{equation*}
        \langle A, B \rangle =\frac{1}{2}\tr(A^T B). 
    \end{equation*}
    If $\{e_i\}_{i=1}^n$ is an orthonormal basis of $V$, then $\{e_i \wedge e_j\}_{1\leq i<j\leq n}$ is an orthonormal basis of $\Lambda^2(V)$. 
  \end{itemize}

\subsection{The Curvature Operator of the Second Kind}
Given $R\in S^2_B(\Lambda^2(V))$, the induced self-adjoint operator $\hat{R}:\Lambda^2 (V) \to \Lambda^2(V)$ given by 
    \begin{equation*}
        \hat{R}(\omega)_{ij}=\frac{1}{2}\sum_{k,l=1}^n R_{ijkl}\omega_{kl}, 
    \end{equation*}
is called the curvature operator (or the curvature operator of the first kind by Nishikawa \cite{Nishikawa86}).
The most famous result concerning $\hat{R}$ is perhaps the differentiable sphere theorem stating that a closed Riemannian manifold with two-positive curvature operator is diffeomorphic to a spherical space form. This is due to Hamilton \cite{Hamilton82} in dimension three, Hamilton \cite{Hamilton86} and Chen \cite{Chen91} in dimension four, and B\"ohm and Wilking \cite{BW08} in all higher dimensions. Rigidity results for closed manifolds with two-nonnegative curvature operator are obtained by \cite{Hamilton86} in dimension three, Hamilton \cite{Hamilton86} and Chen \cite{Chen91} in dimension four, and Ni and Wu \cite{NW07} in all higher dimensions. For other important results regarding $\hat{R}$, see for example \cite{Meyer71}, \cite{GM75}, \cite{Tachibana74}, \cite{PW21} and the references therein. 

By the symmetries of $R\in S^2_B(\Lambda^2(V))$ (not including the first Bianchi identity), $R$ also induces a self-adjoint operator $\overline{R}:S^2(V) \to S^2(V)$ via 
\begin{equation*}
    \overline{R}(\vp)_{ij}=\sum_{k,l=1}^n R_{iklj}\vp_{kl}. 
\end{equation*}
However, the nonnegativity of this operator is too strong in the sense that $\overline{R}: S^2(V) \to S^2(V)$ is nonnegative if and only if $R=0$. 
Therefore, one usually considers the restriction of $\overline{R}$ to the space of traceless symmetric two-tensors, i.e., the induced symmetric bilinear form $\mathring{R}:S^2_0(V)\times S^2_0(V) \to \R$ given by
    \begin{equation*}
        \mathring{R}(\vp,\psi)=\sum_{i,j,k,l=1}^n R_{ijkl}\vp_{il}\psi_{jk}.
   \end{equation*}
Following Nishikawa's terminology \cite{Nishikawa86}, we call the symmetric bilinear form $\mathring{R}$ \textit{the curvature operator of the second kind}. 

The action of the Riemann curvature tensor on symmetric two-tensors indeed has a long history. It appeared for K\"ahler manifolds in the study of the deformation of complex analytic structures by Calabi and Vesentini \cite{CV60}. They introduced the self-adjoint operator $\xi_{\a \b} \to R^{\rho}_{\ \a\b}{}^{\sigma} \xi_{\rho \sigma}$ from $S^2(T^{1,0}_p M)$ to itself, and computed the eigenvalues of this operator on Hermitian symmetric spaces of classical type, with the exceptional ones handled shortly after by Borel \cite{Borel60}. In the Riemannian setting, the operator $\overline{R}$ arises naturally in the context of deformations of Einstein structure in Berger and Ebin \cite{BE69} (see also \cite{Koiso79a, Koiso79b} and \cite{Besse08}). 
In addition, it appears in the Bochner-Weitzenb\"ock formulas for symmetric two-tensors (see for example \cite{MRS20}), for differential forms in \cite{OT79}, and for Riemannian curvature tensors in \cite{Kashiwada93}.  
In another direction, curvature pinching estimates for $\overline{R}$ were studied by Bourguignon and Karcher \cite{BK78}, and they calculated eigenvalues of $\overline{R}$ on the complex projective space with the Fubini-Study metric and the quaternionic projective space with its canonical metric. 
Nevertheless, the operators $\overline{R}$ and $\mathring{R}$ are significantly less investigated than $\hat{R}$. 

Let $N=\dim(S^2_0(V))=\frac{(n-1)(n+2)}{2}$ and $\{\vp_i\}_{i=1}^N$ be an orthonormal basis of $S^2_0(V)$. The $N\times N$ matrix $\mathring{R}(\vp_i, \vp_j)$ is called the matrix representation of $\mathring{R}$ with respect to the orthonormal basis $\{\vp_i\}_{i=1}^N$. 
The eigenvalues of $\mathring{R}$ refer to the eigenvalues of any of its matrix representations. 
Note that the eigenvalues of $\mathring{R}$ are independent of the choices of the orthonormal bases. 

For a positive integer $1\leq k \leq N$, we say $R\in S^2_B(\Lambda^2(V))$ has $k$-nonnegative curvature operator of the second kind if the sum of the smallest $k$ eigenvalues of $\mathring{R}$ is nonnegative. 
This was extended to all $k\in [1,N]$ in \cite{Li22JGA} as follows. 
\begin{definition}
Let $N=\frac{(n-1)(n+2)}{2}$ and $\a \in [1, N]$. 
\begin{enumerate}
    \item We say $R\in S^2_B(\Lambda^2(V))$ has $\a$-nonnegative curvature operator of the second kind ($\mathring{R}$ is $\a$-nonnegative for short) if for any orthonormal basis $\{\vp_i\}_{i=1}^{N}$ of $S^2_0(V)$, it holds that 
\begin{equation*}\label{eq def R}
    \sum_{i=1}^{\lfloor \a \rfloor} \mathring{R}(\vp_i,\vp_i) +(\a -\lfloor \a \rfloor) \mathring{R}(\vp_{\lfloor \a \rfloor+1},\vp_{\lfloor \a \rfloor+1}) \geq  0. 
\end{equation*}
If the inequality is strict, then $R$ is said to have $\a$-positive curvature operator of the second kind ($\mathring{R}$ is $\a$-positive for short).
 \item We say $R\in S^2_B(\Lambda^2(V))$ has $\a$-nonpositive (respectively, $\a$-negative) curvature operator of the second kind if $-R$ has $\a$-nonnegative (respectively, $\a$-positive) curvature operator of the second kind.
\end{enumerate} 
\end{definition}
Note that when $\a=k$ is an integer, this agrees with the usual definition. We always omit $\a$ when $\a=1$. 
Clearly, $\a$-nonnegativity of $\mathring{R}$ implies $\b$-nonnegativity of $\mathring{R}$ if $\a \leq \b$. The same holds for positivity, negativity, and nonpositivity. 

\begin{definition}
A Riemannian manifold $(M^n,g)$ is said to have $\a$-nonnegative (respectively, $\a$-positive, $\a$-nonpositive, $\a$-negative) curvature operator of the second kind if $R_p \in S^2_B(\Lambda^2 T_pM)$ has $\a$-nonnegative (respectively, $\a$-positive, $\a$-nonpositive, $\a$-negative) curvature operator of the second kind for each $p\in M$.    
\end{definition}

The generalized definition is motivated by geometric examples. For instance, $\S^{n-1}\times \S^1$ has $\a$-nonnegative curvature operator of the second kind for any $\a \geq \left(n+\frac{n-2}{n}\right)$, but not for any $\a< \left(n+\frac{n-2}{n}\right)$. Another example is $(\CP^m,g_{FS})$, whose curvature operator of the second kind is $\a$-nonnegative for any $\a \geq \frac{3}{2}(m^2-1)$, but not for any $\a < \frac{3}{2}(m^2-1)$.

\subsection{K\"ahler Algebraic Curvature Operators}
Throughout this subsection, let $(V,g,J)$ be a complex Euclidean vector space of complex dimension $m\geq 1$. 
In other words, $(V,g)$ is a real Euclidean vector space of real dimension $2m$ and $J:V\to V$ is an endomorphism of $V$ satisfying the following two properties: 
\begin{enumerate}
    \item $J^2 =-\id $ on $V$,
    \item $g(X,Y)=g(JX,JY)$ for all $X,Y\in V$.
\end{enumerate} 
$J$ is called a complex structure on $V$.

\begin{definition}\label{def Kahler algebraic curvature operator}
$R\in S^2_B(\Lambda^2 V )$ is called a K\"ahler algebraic curvature operator if it satisfies 
\begin{equation*}
    R(X,Y,Z,W)=R(X,Y,JZ,JW),
\end{equation*}
for all $X,Y,Z,W \in V$. 
\end{definition}

Note that a K\"ahler algebraic curvature operator $R$ satisfies 
\begin{eqnarray}\label{eq R J-inv}
    &&R(X,Y,Z,W)=R(JX,JY,Z,W)\\ &=& R(X,Y,JZ,JW)=R(JX,JY,JZ,JW),\nonumber
\end{eqnarray}
and 
\begin{equation}\label{eq bisectional curvature}
    R(X,JX,Y,JY)=R(X,Y,X,Y) + R(X,JY,X,JY),
\end{equation}
for all $X,Y,Z,W \in V$. 
\eqref{eq R J-inv} follows from the symmetries of $R$ and \eqref{eq bisectional curvature} is a consequence of the first Bianchi identity and \eqref{eq R J-inv}.
The expression on the left-hand side of \eqref{eq bisectional curvature} is called bisectional curvature (see \cite{GK67}), holomorphic sectional curvature if $X=Y$ (see for instance \cite{YZ19}), and orthogonal bisectional curvature if $g(X,Y)=g(X,JY)=0$ (see for example \cite{LN20}). 
We will use \eqref{eq R J-inv} and \eqref{eq bisectional curvature} frequently in the rest of this paper. 

In view of \eqref{eq R J-inv} and \eqref{eq bisectional curvature}, the Ricci tensor of a K\"ahler algebraic curvature operator $R$ is given by 
\begin{equation}\label{eq Ricci curvature}
    \Ric(X,Y)=\sum_{i=1}^m R(X,JY,e_i, Je_i),
\end{equation}
and the scalar curvature of $R$, denoted by $S$, is given by 
\begin{equation}\label{eq scalar curvature}
    S=2\sum_{i,j=1}^m R(e_i,Je_i,e_j,Je_j),
\end{equation}
where $\{e_1,\cdots, e_m, Je_1, \cdots, Je_m\}$ is an orthonormal basis of $V$.

Next, we recall some definitions. 
\begin{definition}
A K\"ahler algebraic curvature operator $R$ is said to have 
\begin{enumerate}
\item nonnegative holomorphic sectional curvature if for any $X\in V$, \begin{equation*}
    R(X,JX,X,JX)\geq 0.
\end{equation*}
\item nonnegative orthogonal bisectional curvature if for any $X, Y\in V$ with $g(X, Y)=g(X,JY)=0$, 
\begin{equation*}
    R(X,JX,Y,JY)\geq 0.
\end{equation*}
\item nonnegative orthogonal Ricci curvature if for any $0\neq X\in V$,
\begin{equation*}
    \Ric^\perp(X,X):=\Ric(X,X)-R(X,JX,X,JX)/|X|^2 \geq 0
\end{equation*}
\end{enumerate} 
\end{definition}
Analogously, one defines the positivity, negativity, and non-positivity of holomorphic sectional curvature, orthogonal bisectional curvature, and orthogonal Ricci curvature. 
Finally, a K\"ahler manifold $(M^m,g,J)$ is said to satisfy a curvature condition if $R_p \in S^2_B(\Lambda^2 T_pM)$ satisfies the curvature condition at every $p\in M$.

\section{Identities}

In this section, we collect some identities that will be frequently used in subsequent sections. 
Many of them have been used explicitly or implicitly in earlier works such as \cite{OT79}, \cite{CGT21}, \cite{Li21, Li22JGA, Li22PAMS}, and \cite{NPW22}. 

\begin{lemma}\label{lemma 3.0}
Let $(V,g)$ be a real Euclidean vector space of dimension $n \geq 2$ and $\{e_i\}_{i=1}^n$ be an orthonormal basis of $V$. Then we have
\begin{equation}\label{eq g on basis}
   \langle e_i \odot e_j ,e_k \odot e_l \rangle = 2(\delta_{ik}\delta_{jl}+\delta_{il}\delta_{jk}),
\end{equation}
and
\begin{equation}\label{eq R on basis}
        \mathring{R}(e_i \odot e_j ,e_k \odot e_l)= 2(R_{iklj}+R_{ilkj}),
    \end{equation}
for all $1\leq i,j,k,l \leq n$. 
\end{lemma}
\begin{proof}
Using $(e_i\otimes e_j)(e_k,e_l)=\delta_{ik}\delta_{jl}$, we compute that
\begin{eqnarray*}
    && \langle e_i \odot e_j ,e_k \odot e_l \rangle  \\
    &=& \sum_{p,q=1}^n (e_i \odot e_j)(e_p,e_q) \cdot (e_k \odot e_l)(e_p,e_q) \\
    &=& \sum_{p,q=1}^n(\delta_{ip}\delta_{jq}+\delta_{iq}\delta_{jp})(\delta_{kp}\delta_{lq}+\delta_{kq}\delta_{lp}) \\
    &=& 2(\delta_{ik}\delta_{jl}+\delta_{il}\delta_{jk}).
\end{eqnarray*}
This proves \eqref{eq g on basis}. 
To prove \eqref{eq R on basis}, we calculate that 
\begin{eqnarray*}
&& \mathring{R}(e_i \odot e_j ,e_k \odot e_l) \\
&=&  \sum_{p,q, r, s=1}^n R_{prsq}\cdot (e_i \odot e_j)(e_p,e_q) \cdot (e_k \odot e_l)(e_r,e_s)  \\
&=& \sum_{p,q, r, s=1}^n R_{prsq}(\delta_{ip}\delta_{jq}+\delta_{iq}\delta_{jp})(\delta_{kr}\delta_{ls}+\delta_{ks}\delta_{lr})   \\
&=& R_{iklj}+R_{ilkj}+R_{jkli}+R_{jlki}\\
&=& 2(R_{iklj}+R_{ilkj}).
\end{eqnarray*}
\end{proof} 

\begin{lemma}\label{lemma ijkl}
    Let $\{e_i,e_j,e_k,e_l\}$ be an orthonormal four-frame in a Euclidean vector space $(V,g)$ of dimension $n\geq 4$. Define the following traceless symmetric two-tensors: 
    \begin{eqnarray*}
        h^{\pm}_1 &=& \frac{1}{2}\left(e_i\odot e_j \pm e_k \odot e_l \right), \\
        h_2 &=& \frac{1}{4}\left(e_i\odot e_i +e_j \odot e_j -e_k \odot e_k -e_l \odot e_l \right) .
    \end{eqnarray*}
Then we have $\|h^{\pm}_1\|=\|h_2\|=1$ and 
\begin{eqnarray*}
    \mathring{R}(h^{\pm}_1,h^{\pm}_1) &=& \frac{1}{2} \left( R_{ijij} +R_{klkl} \right) \pm R_{iklj} \pm R_{ilkj}, \\
    \mathring{R}(h_2,h_2) &=& \frac{1}{2}\left( -R_{ijij}-R_{klkl}+R_{ikik}+R_{ilil}+R_{jkjk}+R_{jljl}\right).
\end{eqnarray*}
\end{lemma}
\begin{proof}
One easily verifies using \eqref{eq g on basis} that $\|h^{\pm}_1\|=\|h_2\|=1$. 
We compute that
\begin{eqnarray*}
    && 4 \mathring{R}(h^{\pm}_1,h^{\pm}_1) \\
    &=&\mathring{R}(e_i\odot e_j \pm e_k \odot e_l, e_i\odot e_j \pm e_k \odot e_l) \\
    &=& \mathring{R}(e_i\odot e_j , e_i\odot e_j) +\mathring{R}(e_k \odot e_l, e_k \odot e_l) \pm 2 \mathring{R}(e_i\odot e_j, e_k \odot e_l) \\
    &=& 2R_{ijij} +2R_{klkl} \pm 4 \left(R_{iklj}+R_{ilkj} \right),
\end{eqnarray*}
where we have used \eqref{eq R on basis} in getting the last line. 

Using \eqref{eq R on basis} again, we obtain that
\begin{eqnarray*}
    && 16 \mathring{R}(h_2,h_2) \\
    &=&\mathring{R}(e_i\odot e_i+e_j\odot e_j, e_i \odot e_i+ e_j\odot e_j) \\
     &&+\mathring{R}(e_k\odot e_k+e_l\odot e_l, e_k \odot e_k+ e_l\odot e_l) \\
     &&-2\mathring{R}(e_i\odot e_i+e_j\odot e_j, e_k \odot e_k+ e_l\odot e_l) \\
    &=& 2 \mathring{R}(e_i\odot e_i, e_j\odot e_j) +2 \mathring{R}(e_k\odot e_k, e_l\odot e_l) \\ 
    && -2 \mathring{R}(e_i\odot e_i, e_k\odot e_k) -2 \mathring{R}(e_i\odot e_i, e_l\odot e_l) \\ 
     && -2 \mathring{R}(e_j\odot e_j, e_k\odot e_k) -2 \mathring{R}(e_j\odot e_j, e_l\odot e_l) \\ 
    &=& -8(R_{ijij} +R_{klkl}) +8\left(R_{ikik}+R_{ilil}+R_{jkjk}+R_{jljl} \right).
\end{eqnarray*}
\end{proof}

\begin{lemma}\label{lemma ij}
Let $\{e_i,e_j\}$ be an orthonormal two-frame in a Euclidean vector space $(V,g)$ of dimension $n\geq 2$. Define the following traceless symmetric two-tensors:
\begin{eqnarray*}
        h_3 &=& \frac{1}{2\sqrt{2}}\left(e_i\odot e_i -e_j \odot e_j  \right), \\
        h_4 &=& \frac{1}{\sqrt{2}} e_i \odot e_j.
    \end{eqnarray*}
Then we have $\|h_3\|=\|h_4\| =1$ and 
\begin{equation*}
    \mathring{R}(h_3,h_3)= \mathring{R}(h_4,h_4) = R_{ijij}.
\end{equation*}
\end{lemma}
\begin{proof}
This is a straightforward computation using \eqref{eq g on basis} and \eqref{eq R on basis}. 
\end{proof}

\begin{lemma}
Let $\{e_1, \cdots, e_m, Je_1, \cdots, Je_m\}$ be an orthonormal basis of a complex Euclidean vector space $(V,g,J)$ of complex dimension $m\geq 1$. Then for any $1 \leq i, j \leq m$, we have
\begin{equation}\label{eq R iiJiJi}
    \mathring{R}(e_i \odot e_i +Je_i \odot Je_i, e_j \odot e_j +Je_j \odot Je_j) =-8R(e_i,Je_i,e_j,Je_j).
\end{equation}
\end{lemma}
\begin{proof}
This follows from a routine computation using \eqref{eq R on basis}, \eqref{eq R J-inv}, and \eqref{eq bisectional curvature} as follows: 
\begin{eqnarray*}  
&& \mathring{R}(e_i \odot e_i +Je_i \odot Je_i, e_j \odot e_j +Je_j \odot Je_j) \\
&=& \mathring{R}(e_i \odot e_i, e_j \odot e_j) +\mathring{R}(e_i \odot e_i, Je_j \odot Je_j) \\
&& +\mathring{R}(Je_i \odot Je_i, e_j \odot e_j)+\mathring{R}(Je_i \odot Je_i, Je_j \odot Je_j) \\
&=& -4R(e_i,e_j,e_i,e_j)-4R(e_i,Je_j,e_i,Je_j)\\
&& -4R(Je_i,e_j,Je_i,e_j)-4R(Je_i,Je_j,Je_i,Je_j) \\
&=& -8R(e_i,e_j,e_i,e_j)-8R(e_i,Je_j,e_i,Je_j)\\
&=& -8R(e_i,Je_i,e_j,Je_j).
\end{eqnarray*}   
\end{proof}

The author observed in \cite[Proposition 4.1]{Li21} (see also \cite[Proposition 1.2]{NPW22}) that the trace of $\mathring{R}$ is equal to $\frac{n+2}{2n}S$, where $S$ denotes the scalar curvature. That is to say, if $\{\vp_i\}_{i=1}^N$ is an orthonormal basis of $S^2_0(V)$, then 
\begin{equation}\label{trace R}
    \sum_{i=1}^N \mathring{R}(\vp_i,\vp_i) =\frac{n+2}{2n}S.
\end{equation}
This implies that 
\begin{lemma}\label{lemma R and S}
$R\in S^2_B(\Lambda^2 V)$ has $\frac{(n-1)(n+2)}{2}$-nonnegative (respectively, $\frac{(n-1)(n+2)}{2}$-nonpositive) curvature operator of the second kind if and only if $R$ has nonnegative (respectively, nonpositive) scalar curvature $S$.   
\end{lemma}
\begin{lemma}\label{S flat implies R flat}
Suppose that $R\in S^2_B(\Lambda^2 V)$ has $\a$-nonnegative (respectively, $\a$-nonpositive) curvature operator of the second kind for some $\a < \frac{(n-1)(n+2)}{2}$. If $S=0$, then $R=0$. 
\end{lemma}

\section{An orthonormal basis for $S^2_0(V)$}
Below we construct an orthonormal basis for $S^2_0(V)$ on a complex Euclidean vector space $(V,g,J)$. 

\begin{lemma}\label{lemma basis +}
Let $\{e_1, \cdots, e_m, Je_1, \cdots, Je_m\}$ be an orthonormal basis of a complex Euclidean vector space $(V,g,J)$. Let 
\begin{equation*}
    E^+=\spn\{u \odot v -Ju \odot Jv: u,v \in V \}.
\end{equation*}
Define 
\begin{eqnarray*}
\vp^+_{ij}  &=& \frac{1}{2} \left( e_i \odot e_j - Je_i \odot Je_j \right), \text{ for } 1 \leq i < j \leq m, \\
\psi^{+}_{ij}  &=& \frac{1}{2} \left( e_i \odot J e_j + Je_i \odot e_j \right), \text{ for } 1 \leq i < j \leq m, \\
    \theta_{i} &=& \frac{1}{2\sqrt{2}} \left( e_i \odot e_i -Je_i \odot Je_i \right), \text{ for } 1\leq i \leq m, \\
     \theta_{m+i} &=& \frac{1}{\sqrt{2}} e_i \odot J e_i, \text{ for } 1\leq i \leq m.
\end{eqnarray*}
Then 
\begin{equation}\label{eq basis}
   \mathcal{E}^+= \{\vp^{+}_{ij}\}_{1\leq i < j \leq m} \cup\{\psi^{+}_{ij}\}_{1\leq i < j \leq m} \cup \{\theta_i \}_{i=1}^{2m} 
\end{equation}
forms an orthonormal basis of $E^+$. In particular, $\dim(E^+)=m(m+1)$. 
\end{lemma}
\begin{proof}
Clearly, $\mathcal{E}^+ \subset E^+$. Using \eqref{eq g on basis}, one verifies
that $\mathcal{E}^+$ is an orthonormal subset of $E^+$. The statement $\mathcal{E}^+$ spans $E^+$ follows from the following observation. If 
\begin{eqnarray*}
    u &=& \sum_{i=1}^m x_i e_i +\sum_{i=1}^m y_i Je_i , \\
    v &=& \sum_{i=1}^m z_i e_i +\sum_{i=1}^m w_i Je_i,
\end{eqnarray*}
then 
\begin{eqnarray*}
    && u\odot v -Ju \odot Jv \\
    &=& \sum_{i,j=1}^m (x_iz_j-y_iw_j)(e_i\odot e_j -Je_i \odot Je_j)  \\
    &&+ \sum_{i,j=1}^m (x_iw_j+y_iz_j)(e_i\odot Je_j + Je_i \odot e_j)  \\
    &=& 4\sum_{1\leq i < j\leq m} (x_iz_j-y_iw_j)\vp^+_{ij}  +4\sum_{1\leq i < j\leq m} (x_iw_j+y_iz_j)\psi^+_{ij}  \\
    && + 2\sqrt{2} \sum_{i=1}^m (x_iz_i-y_iw_i)\theta_{i} + 2\sqrt{2} \sum_{i=1}^{m} (x_iw_i+y_iz_i)\theta_{m+i}.
\end{eqnarray*}
Thus, $\mathcal{E}^+$ forms an orthonormal basis of $E^+$ and $\dim(E^+)=m(m+1)$. 
\end{proof}

\begin{lemma}\label{lemma basis -}
Let $\{e_1, \cdots, e_m, Je_1, \cdots, Je_m\}$ be an orthonormal basis of a complex Euclidean vector space $(V,g,J)$. Let $E^-=(E^+)^\perp$ be the orthogonal complement of $E^+$, where $E^+$ is the subspace of $S^2_0(V)$ defined in Lemma \ref{lemma basis +}. 
Define
\begin{eqnarray*}
\vp^{-}_{ij}  &=& \frac{1}{2} \left( e_i \odot e_j + Je_i \odot Je_j \right), \text{ for } 1 \leq i < j \leq m, \\
\psi^{-}_{ij}  &=& \frac{1}{2} \left( e_i \odot J e_j - Je_i \odot e_j \right), \text{ for } 1 \leq i < j \leq m,
\end{eqnarray*}
and
\begin{eqnarray*}
    \eta_k &=& \frac{k}{\sqrt{8k(k+1)}}  (e_{k+1}\odot e_{k+1} +Je_{k+1} \odot Je_{k+1}) \\
    &&  - \frac{1}{\sqrt{8k(k+1)}}\sum_{i=1}^{k}(e_i \odot e_i +Je_i \odot Je_i),
\end{eqnarray*} 
for $1\leq k \leq m-1$. Then 
\begin{equation}
    \mathcal{E}^-=\{\vp^{-}_{ij}\}_{1\leq i < j \leq m} \cup \{\psi^{-}_{ij}\}_{1\leq i < j \leq m} \cup \{\eta_k\}_{k=1}^{m-1}
\end{equation}
forms an orthonormal basis of $E^-$. In particular, $\dim(E^-)=m^2-1$. 
\end{lemma}

\begin{proof}
Since $S^2_0(V)=E^+\oplus E^-$, we have that
\begin{eqnarray*}
    \dim(E^-) &=& \dim(S^2_0(V))-\dim(E^+) \\
    &=& (2m-1)(m+1) -m(m+1)\\
    &=&m^2-1.
\end{eqnarray*}
As the number of traceless symmetric two-tensors in $\mathcal{E}^-$ is equal to $\dim(E^-)$, it suffices to verify that $\mathcal{E}^+ \cup \mathcal{E}^-$ is an orthonormal basis of $S^2_0(V)$, which is a straightforward computation using \eqref{eq g on basis}. 
\end{proof}

We remark that on $(\CP^m, g_{FS})$, $E^+$ is the eigenspace associated with the eigenvalue $4$, and $E^-$ is the eigenspace associated with the eigenvalue $-2$. 
The subspace $E^-$ is spanned by traceless symmetric two-tensors of the form 
$u \odot v +Ju \odot Jv$ with $g(u,v)=0$ and $u\odot u +Ju \odot Ju -v\odot v -Jv\odot Jv$ with $g(u,u)=g(v,v)$. See \cite[page 84]{BK78}.

The next step is to calculate the matrix representation of $\mathring{R}$ with respect to the orthonormal basis $\mathcal{E}^+ \cup \mathcal{E}^-$ for $S^2_0(V)$. We only need the diagonal elements of this matrix. 
\begin{lemma}\label{lemma R basis +}
For the basis $\mathcal{E}^+$ of $E^+\subset S^2_0(V)$ defined in Lemma \ref{lemma basis +}, we have
\begin{equation}\label{eq R positive basis vp psi}
     \mathring{R}(\vp^+_{ij}, \vp^+_{ij}) =\mathring{R}(\psi^+_{ij}, \psi^+_{ij})=  2R(e_i, Je_i, e_j, Je_j)
\end{equation}
for $1\leq i < j \leq m$, and 
\begin{equation}\label{eq R positive basis alpha}
    \mathring{R}(\theta_{i}, \theta_{i}) =\mathring{R}(\theta_{m+i}, \theta_{m+i})= R(e_i,Je_i, e_i, Je_i), 
\end{equation}
for $1\leq i \leq m$.
Moreover, 
\begin{eqnarray}\label{R sum positive eigenvalues}
     \sum_{1\leq i< j \leq m} \left( \mathring{R}(\vp^+_{ij}, \vp^+_{ij}) + \mathring{R}(\psi^+_{ij}, \psi^+_{ij}) \right) + \sum_{i=1}^{2m} \mathring{R}(\theta_{i}, \theta_{i}) 
   = S. 
\end{eqnarray}
\end{lemma}
\begin{proof}
Applying Lemma \ref{lemma ijkl} to the orthonormal four-frame $\{e_i,e_j,Je_i,Je_j\}$ yields 
\begin{eqnarray*}
    \mathring{R}(\vp^+_{ij}, \vp^+_{ij}) &=&  \frac{1}{2} \left( R(e_i,e_j,e_i,e_j) +R(Je_i,Je_j,Je_i,Je_j) \right) \\
    && -R(e_i, Je_i, Je_j, e_j) - R(e_i, Je_j, Je_i, e_j) \\
    &=& R(e_i,e_j,e_i,e_j) +R(e_i, Je_i, e_j, Je_j) + R(e_i, Je_j, e_i, Je_j)\\
    &=& 2R(e_i, Je_i, e_j, Je_j),
\end{eqnarray*}
where we have used \eqref{eq R J-inv} and \eqref{eq bisectional curvature}.
Similarly, we get 
\begin{eqnarray*}
    \mathring{R}(\psi^+_{ij}, \psi^+_{ij}) &=&  \frac{1}{2} \left( R(e_i,Je_j,e_i,Je_j) +R(Je_i,e_j,Je_i,e_j) \right) \\
    && +R(e_i, Je_i, e_j, Je_j) + R(e_i, e_j, Je_i, Je_j) \\
    &=&  R(e_i, Je_j, e_i, Je_j)+R(e_i, Je_i, e_j, Je_j) +R(e_i,e_j,e_i,e_j) \\
    &=& 2R(e_i, Je_i, e_j, Je_j).
\end{eqnarray*}
Now, \eqref{eq R positive basis vp psi} is proved. 

For $1\leq i \leq m$, we apply Lemma \ref{lemma ij} to the orthonormal two-frame $\{e_i,Je_i\}$ and get 
\begin{equation*}
 \mathring{R}(\theta_{i}, \theta_{i}) =\mathring{R}(\theta_{m+i}, \theta_{m+i})= R(e_i,Je_i, e_i, Je_i).
\end{equation*}
This proves  \eqref{eq R positive basis alpha}. 

Finally, \eqref{R sum positive eigenvalues} follows from \eqref{eq R positive basis vp psi}, \eqref{eq R positive basis alpha}, and \eqref{eq scalar curvature}. 
\end{proof}

\begin{lemma}\label{lemma R basis -}
For the basis $\mathcal{E}^-$ of $E^-\subset S^2_0(V)$ defined in Lemma \ref{lemma basis -}, we have
\begin{eqnarray}\label{eq sum of negative eigenvaleus}
    \sum_{1\leq i <j \leq m} \left( \mathring{R}(\vp^-_{ij}, \vp^-_{ij}) + \mathring{R}(\psi^-_{ij}, \psi^-_{ij}) \right) +\sum_{k=1}^{m-1} \mathring{R}(\eta_k,\eta_k) 
    = -\frac{m-1}{2m} S. 
\end{eqnarray}
\end{lemma}

\begin{proof}
Since $\mathcal{E}^+ \cup \mathcal{E}^-$ is an orthonormal basis for $S^2_0(V)$, 
\eqref{eq sum of negative eigenvaleus} follows immediately from and \eqref{R sum positive eigenvalues}, \eqref{trace R} and \eqref{eq scalar curvature}. 
\end{proof}

\section{Flatness}
We prove Theorem \ref{thm flat} in this section. The key ingredient is 
\begin{proposition}\label{prop flat}
    Let $R$ be a K\"ahler algebraic curvature operator on a complex Euclidean vector space $(V,g,J)$ of complex dimension $m\geq 2$. 
    \begin{enumerate}
        \item If $R$ has $\frac{3}{2}(m^2-1)$-nonnegative (respectively, $\frac{3}{2}(m^2-1)$-nonpositive) curvature operator of the second kind, then $R$ has constant nonnegative (respectively, nonpositive) holomorphic sectional curvature. 
       \item If $R$ has $\a$-nonnegative (respectively, $\a$-nonpositive) curvature operator of the second kind for some $\a < \frac{3}{2}(m^2-1)$, then $R$ is flat.
    \end{enumerate}
\end{proposition}

We need an elementary lemma.
\begin{lemma}\label{lemma average}
Let $N$ be a positive integer and $A$ be a collection of $N$ real numbers.  Denote by $a_i$ the $i$-th smallest number in $A$ for $1\leq i \leq N$.  
Define a function $f(A,x)$ by 
   \begin{equation*}
       f(A,x)=\sum_{i=1}^{\lfloor x \rfloor} a_i +(x-\lfloor x \rfloor) a_{\lfloor x \rfloor+1}, 
   \end{equation*}
   for $x\in [1,N]$. Then we have
   \begin{equation}\label{eq function}
       f(A,x) \leq x \bar{a}, 
   \end{equation}
   where $\bar{a}:=\frac{1}{N}\sum_{i=1}^N a_i$ is the average of all numbers in $A$. 
   Moreover, the equality holds for some $x\in [1,N)$ if and only if $a_i=\bar{a}$ for all $1\leq i\leq N$. 
\end{lemma}
\begin{proof}
We first show that Lemma \ref{lemma average} holds when $x=k$ is an integer. This is obvious if $k=1$ or $k=N$. 
If $2\leq k \leq N-1$, we have 
\begin{eqnarray*}
   N(f(A,k)-k\bar{a}) 
  &=& N\sum_{i=1}^{k} a_i - k \sum_{i=1}^N a_i \\
  &=& (N-k)\sum_{i=1}^{k} a_i -k \sum_{i=k+1}^N a_i .
\end{eqnarray*}
Note that $\sum_{i=1}^{k} a_i \leq ka_{k+1}$ with equality if and only if $a_1=\cdots =a_{k+1}$ and $\sum_{i=k+1}^N a_i \geq (N-k)a_{k+1}$ with equality if and only if $a_{k+1}=\cdots =a_{N}$. So, we get 
\begin{equation*}
    N(f(A,k)-k\bar{a}) 
  \leq  (N-k)k a_{k+1} -k(N-k)a_{k+1}  = 0.
\end{equation*}
Moreover, the equality holds if and only if $a_i=\bar{a}$ for all $1\leq i\leq N$. 

Next, for $x \in (1,N)$, we have 
\begin{eqnarray*}
  f(A,x) &=& \sum_{i=1}^{\lfloor x \rfloor} a_i +(x-\lfloor x \rfloor) a_{\lfloor x \rfloor+1} \\
  &=& (x-\lfloor x \rfloor)\sum_{i=1}^{\lfloor x \rfloor+1} a_i +(1-x+\lfloor x \rfloor) \sum_{i=1}^{\lfloor x \rfloor} a_i \\
  &=& (x-\lfloor x \rfloor) f(A, \lfloor x \rfloor+1)  +(1-x+\lfloor x \rfloor) f(A,\lfloor x \rfloor) \\
  &\leq & (x-\lfloor x \rfloor) (\lfloor x \rfloor+1)\bar{a}  +(1-x+\lfloor x \rfloor) \lfloor x \rfloor \bar{a} \\
  &=& x \bar{a}. 
\end{eqnarray*}
Here we have used $f(A, \lfloor x \rfloor+1) \leq (\lfloor x \rfloor+1) \bar{a}$ and $f(A,\lfloor x \rfloor) \leq \lfloor x \rfloor \bar{a}$ in getting the inequality step.

If the equality holds in \eqref{eq function} for $x\in [1,N)$, then we must have $f(A, \lfloor x \rfloor+1) = (\lfloor x \rfloor+1) \bar{a}$ and $f(A,\lfloor x \rfloor) = \lfloor x \rfloor \bar{a}$, which implies $a_i=\bar{a}$ for all $1\leq i\leq N$. This completes the proof. 
\end{proof}

We are now ready to prove Proposition \ref{prop flat}. 
\begin{proof}[Proof of Proposition \ref{prop flat}]
(1). Let $\{e_1, \cdots, e_m, Je_1, \cdots, Je_m\}$ be an orthonormal basis of $V$ and let $\mathcal{E}^+\cup \mathcal{E}^-$ be the orthonormal basis of $S^2_0(V)$ constructed in Lemma \ref{lemma basis +} and Lemma \ref{lemma basis -}. 

Let $A$ be the collection of the values $\mathring{R}(\theta_i,\theta_i)$ for $1\leq i \leq 2m$, $\mathring{R}(\vp^{+}_{ij},\vp^{+}_{ij})$ and $\mathring{R}(\psi^{+}_{ij},\psi^{+}_{ij})$ for $1\leq i < j \leq m$.
By \eqref{R sum positive eigenvalues}, the sum of all values in $A$ is equal to $S$, and $\bar{a}$, the average of all values in $A$, is given by 
\begin{equation*}
    \bar{a}=\frac{S}{m(m+1)}.
\end{equation*}
By Lemma \ref{lemma average}, we have
\begin{equation}\label{eq f1}
    f\left(A,\frac{1}{2}(m^2-1)\right) \leq \frac{1}{2}(m^2-1) \bar{a} =\frac{m-1}{2m}S,
\end{equation}
where $f(A,x)$ is the function defined in Lemma \ref{lemma average}. 

If $R$ has $\frac{3}{2}(m^2-1)$-nonnegative curvature operator of the second kind, then
\begin{eqnarray*}
 0 \leq \sum_{1\leq i <j \leq m} \left( \mathring{R}(\vp^-_{ij}, \vp^-_{ij}) + \mathring{R}(\psi^-_{ij}, \psi^-_{ij}) \right) +\sum_{k=1}^{m-1} \mathring{R}(\eta_k,\eta_k)  +f\left(A,\frac{1}{2}(m^2-1)\right).
\end{eqnarray*}
Substituting \eqref{eq sum of negative eigenvaleus} into the above inequality yields 
\begin{eqnarray}\label{eq 4.1}
 \frac{m-1}{2m} S   \leq    f\left(A,\frac{1}{2}(m^2-1)\right).
\end{eqnarray}
Therefore, \eqref{eq f1} must hold as equality and we conclude from Lemma \ref{lemma average} that all values in $A$ are equal to $\frac{S}{m(m+1)}$. 
In particular,  
\begin{equation*}
    R(e_i,Je_i,e_i,Je_i)=\frac{S}{m(m+1)}
\end{equation*}
for all $1\leq i \leq m$. 
Since the orthonormal basis $\{e_1,\cdots, e_m, Je_1, \cdots, Je_m\}$ is arbitrary, we have 
\begin{equation*}
    R(X,JX,X,JX)=\frac{S}{m(m+1)}
\end{equation*}
for any unit vector $X\in V$. Hence, $R$ has constant holomorphic sectional curvature. Finally, we notice that $S\geq 0$ by Lemma \ref{lemma R and S}.

If $\mathring{R}$ is $\frac{3}{2}(m^2-1)$ nonpositive,  we apply the proved result to $-R$ and get that $R$ has constant nonpositive holomorphic sectional curvature.  

(2). $R$ has constant holomorphic sectional curvature by part (1). Thus, $R=c R_{\CP^m}$ for some $c\in \R$, where $R_{\CP^m}$ denotes the Riemann curvature tensor of $ (\CP^m, g_{FS})$. 
The desired conclusion follows from the fact that $c R_{\CP^m}$ has $\a$-nonnegative or $\a$-nonpositive curvature operator of the second kind for some $\a < \frac{3}{2}(m^2-1)$ if and only if $c=0$ (see \cite{BK78}). 

Alternatively, one can argue as follows. If $\mathring{R}$is $\left(\frac{3}{2}(m^2-1)-\epsilon\right)$-nonnegative for some $1>\epsilon >0$, then we would get in \eqref{eq 4.1} that
\begin{equation*}
 0   \leq   -\frac{m-1}{2m} S  + f\left(A,\frac{1}{2}(m^2-1)-\epsilon \right) 
 \leq   \frac{-\epsilon \ S}{m(m+1)}. 
\end{equation*}
So, we have $S=0$. By Lemma \ref{S flat implies R flat}, $R=0$. 
Similarly, $\mathring{R}$ is $\a$-nonpositive for some $\a < \frac{3}{2}(m^2-1)$ if and only if $R=0$. 
\end{proof}

We now prove Theorem \ref{thm flat} and its corollaries. 

\begin{proof}[Proof of Theorem \ref{thm flat}]
Theorem \ref{thm flat} follows immediately from Proposition \ref{prop flat} and  Schur's lemma for K\"ahler manifolds (see for instance \cite[Theorem 7.5]{KN69}).
\end{proof}

\begin{proof}[Proof of Corollary \ref{cor 1}]
If $\mathring{R}$ is $\a$-positive or $\a$-negative for some $\a\leq \frac{3}{2}(m^2-1)$, then $R$ must have constant holomorphic sectional curvature by Proposition \ref{prop flat}. Thus, $R=c R_{\CP^m}$ for some $c\in \R$.
This contradicts the fact that $c R_{\CP^m}$ does not have $\a$-positive or $\a$-nonnegative curvature operator of the second kind for any $\a\leq \frac{3}{2}(m^2-1)$. 
\end{proof}

\begin{proof}[Proof of Corollary \ref{cor 3}]
By Theorem \ref{thm flat}, $M$ has constant holomorphic sectional curvature. The conclusion follows from the the classification of complete simply-connected K\"ahler manifolds with constant holomorphic sectional curvature. 
\end{proof}

\section{Orthogonal Bisectional Curvature}
Throughout this section, $\a_m$ is the number defined in \eqref{eq alpha m def}, i.e., 
\begin{equation*}
        \a_m=\frac{3m^3-m+2}{2m}.
\end{equation*}
We restate part (2) of Theorem \ref{thm algebra R} as two propositions. 
\begin{proposition}\label{prop OB}
    Let $R$ be a K\"ahler algebraic curvature operator on a complex Euclidean vector space $(V,g,J)$ of complex dimension $m\geq 2$.  
    If $R$ has $\a_m$-nonnegative (respectively, $\a_m$-positive, $\a_m$-nonpositive, $\a_m$-negative) curvature operator of the second kind, then $R$ has nonnegative (respectively, positive, nonpositive, negative) orthogonal bisectional curvature. 
\end{proposition}
\begin{proposition}\label{prop H}
    Let $R$ be a K\"ahler algebraic curvature operator on a complex Euclidean vector space $(V,g,J)$ of complex dimension $m\geq 2$. If $R$ has $\a_m$-nonnegative (respectively, $\a_m$-positive, $\a_m$-nonpositive, $\a_m$-negative) curvature operator of the second kind, then $R$ has nonnegative (respectively, positive, nonpositive, negative) holomorphic sectional curvature. 
\end{proposition}

\begin{proof}[Proof of Proposition \ref{prop OB}] 
Let $\{e_1, \cdots, e_m, Je_1, \cdots, Je_m\}$ be an orthonormal basis of $V$ and let $\mathcal{E}^+\cup \mathcal{E}^-$ be the orthonormal basis of $S^2_0(V)$ constructed in Lemma \ref{lemma basis +} and Lemma \ref{lemma basis -}. 

Let $A$ be the collection of the values $\mathring{R}(\theta_i,\theta_i)$ for $1\leq i \leq m$. 
By Lemma \ref{lemma R basis +}, $A$ consists of one copy of $R(e_i,Je_i,e_i,Je_i)$ for each $1\leq i\leq m$, and its average $\bar{a}$ is given by 
\begin{equation*}
\bar{a}=\frac{1}{m}\sum_{i=1}^m R(e_i,Je_i,e_i,Je_i).
\end{equation*}

Let $B$ be the collection of the values $\mathring{R}(\vp^{+}_{ij},\vp^{+}_{ij})$ and $\mathring{R}(\psi^{+}_{ij},\psi^{+}_{ij})$ for $1\leq i < j \leq m$ with $(i,j)\neq (1,2)$. According to Lemma \ref{lemma R basis +}, $B$ is made of two copies of $2R(e_i,Je_i,e_j,Je_j)$ for each $1\leq i < j \leq m$ with $(i,j)\neq (1,2)$. 
The average of all values in $B$, denoted by $\bar{b}$, is given by 
\begin{equation*}
    \bar{b}=\frac{4}{(m-2)(m+1)}\left( \sum_{1\leq i< j\leq m}^m R(e_i,Je_i,e_j,Je_j) - R(e_1,Je_1,e_2,Je_2)\right).
\end{equation*}

Next, let $f(A,x)$ and $f(B,x)$ be defined as in Lemma \ref{lemma average}. By Lemma \ref{lemma average}, we have 
\begin{equation}\label{eq a bar}
    f(A,m-1) \leq (m-1)\bar{a}=\frac{m-1}{m}\sum_{i=1}^m R(e_i,Je_i,e_i,Je_i), 
\end{equation}
and 
\begin{eqnarray}\label{eq b bar}
    && f\left(B, \frac{(m-2)(m^2-1)}{2m}\right) 
    \leq  \frac{(m-2)(m^2-1)}{2m} \bar{b}  \\
    &=& 2\frac{m-1}{m}\left( \sum_{1\leq i< j\leq m}^m R(e_i,Je_i,e_j,Je_j) - R(e_1,Je_1,e_2,Je_2)\right). \nonumber
\end{eqnarray}

Note that 
\begin{equation*}
        \a_m=(m^2-1)+2+(m-1)+\frac{(m-2)(m^2-1)}{2m}.
\end{equation*}
If $R$ has $\a_m$-nonnegative curvature operator of the second kind, then we have 
\begin{eqnarray*}\label{eq R nonnegative ob}
 0 &\leq & \sum_{1\leq i <j \leq m} \left( \mathring{R}(\vp^-_{ij}, \vp^-_{ij}) + \mathring{R}(\psi^-_{ij}, \psi^-_{ij}) \right) +\sum_{k=1}^{m-1} \mathring{R}(\eta_k,\eta_k)  \\
 && + \mathring{R}(\vp^+_{12},\vp^+_{12})+\mathring{R}(\psi^+_{12},\psi^+_{12}) \nonumber \\
 &&  + f(A, m-1)  +f\left(B, \frac{(m-2)(m^2-1)}{2m}\right). \nonumber
\end{eqnarray*}
Substituting \eqref{eq sum of negative eigenvaleus}, \eqref{eq R positive basis vp psi}, \eqref{eq a bar}, and \eqref{eq b bar} 
into the above inequality yields 
\begin{eqnarray*}    
 0 &\leq & -\frac{m-1}{2m} S+ 4R(e_1,Je_1,e_2,Je_2) + \frac{m-1}{m}\sum_{i=1}^m R(e_i,Je_i,e_i,Je_i)  \\ && +\frac{2(m-1)}{m}\left( \sum_{1\leq i< j\leq m}^m R(e_i,Je_i,e_j,Je_j) - R(e_1,Je_1,e_2,Je_2)\right) \\
 &=& \frac{2(m+1)}{m} R(e_1,Je_1,e_2,Je_2),
\end{eqnarray*}
where we have used \eqref{eq scalar curvature}. 
The arbitrariness of the orthonormal frame allows us to conclude that $R(e_1,Je_1,e_2,Je_2) \geq 0$ for any orthonormal two-frame $\{e_1,e_2\}$. Hence, $R$ has nonnegative orthogonal bisectional curvature. 

If $R$ has $\a_m$-positive curvature operator of the second kind, then the last two inequalities in the above argument become strict and one concludes that $R$ has positive orthogonal bisectional curvature. 
Applying the results to $-R$ then yields the statements for $\a_m$-negativity and $\a_m$-nonpositivity.
\end{proof}

\begin{proof}[Proof of Proposition \ref{prop H}]
The idea is the same as in the proof of Proposition \ref{prop OB}. The main difference is that we need to separate the terms involving $R(e_1,Je_1,e_1,Je_1)$. 
As before, let $\{e_1, \cdots, e_m, Je_1, \cdots, Je_m\}$ be an orthonormal basis of $V$ and let $\mathcal{E}^+\cup \mathcal{E}^-$ be the orthonormal basis of $S^2_0(V)$ constructed in Lemma \ref{lemma basis +} and Lemma \ref{lemma basis -}. 

Let $A$ be the collection of the values $\mathring{R}(\theta_i,\theta_i)$ for $2 \leq i \leq m$ and $m+2\leq i \leq 2m$, and the values $\mathring{R}(\vp^{+}_{ij},\vp^{+}_{ij})$ and $\mathring{R}(\psi^{+}_{ij},\psi^{+}_{ij})$ for $1\leq i < j \leq m$. 
By Lemma \ref{lemma R basis +}, $A$ consists of two copies of $R(e_i,Je_i,e_i,Je_i)$ for each $2\leq i \leq m$
and two copies of $2R(e_i,Je_i,e_j,Je_j)$ for each $1\leq i < j \leq m$. 
The average of all values in $A$, denoted by $\bar{a}$, is given by 
\begin{eqnarray*}
    \bar{a} &=& \frac{1}{(m-1)(m+2)}\left(2\sum_{i=2}^m R(e_i,Je_i,e_i,Je_i) +4\sum_{1\leq i < j \leq m} R(e_i,Je_i,e_j,Je_j)\right)\\
    &=& \frac{1}{(m-1)(m+2)}\left(S-2R(e_1,Je_1,e_1,Je_1) \right),
\end{eqnarray*}
where we have used \eqref{eq scalar curvature}. 

Let $f(A,x)$ be the function defined in Lemma \ref{lemma average}, then we have by Lemma \ref{lemma average} that
\begin{eqnarray}\label{eq a bar H}
      &&  f\left(A,\frac{(m-1)^2(m+2)}{2m}\right)
     \leq   \frac{(m-1)^2(m+2)}{2m} \bar{a} \\
    & = & \frac{m-1}{2m} \left(S-2R(e_1,Je_1,e_1,Je_1)\right). \nonumber
\end{eqnarray}

Note that 
\begin{equation*}
        \a_m=(m^2-1)+2+\frac{(m-1)^2(m+2)}{2m}.
\end{equation*}
Since $\mathring{R}$ is $\a_m$-nonnegative, we obtain
\begin{eqnarray*}
 0 &\leq & \sum_{1\leq i <j \leq m} \left( \mathring{R}(\vp^-_{ij}, \vp^-_{ij}) + \mathring{R}(\psi^-_{ij}, \psi^-_{ij}) \right) +\sum_{k=1}^{m-1} \mathring{R}(\eta_k,\eta_k)   \\
 && + \mathring{R}(\theta_{1},\theta_{1}) +\mathring{R}(\theta_{m+1},\theta_{m+1}) \\
 &&  +f\left(A,\frac{(m-1)^2(m+2)}{2m}\right).
 \end{eqnarray*}
Substituting \eqref{eq sum of negative eigenvaleus}, \eqref{eq R positive basis alpha}, and \eqref{eq a bar H}
into the above inequality yields 
\begin{eqnarray*}
    0 &\leq & -\frac{m-1}{2m}S +2R(e_1,Je_1,e_1,Je_1) \\
    && +\frac{m-1}{2m} (S-2R(e_1,Je_1,e_1,Je_1)) \\
&=&   \frac{m+1}{m}   R(e_1,Je_1,e_1,Je_1).
\end{eqnarray*}
Since the orthonormal frame $\{e_1, \cdots, e_m, Je_1, \cdots, Je_m\}$ is arbitrary, we conclude that $R$ has nonnegative holomorphic sectional curvature. 
Other statements can be proved similarly. 
\end{proof}

\begin{proof}[Proof of Theorem \ref{thm OB +}]
By Proposition \ref{prop OB}, $M$ has positive orthogonal bisectional curvature. 
$M$ is biholomorphic to $\CP^m$ by \cite{Chen07} and \cite{GZ10} (or \cite{Wilking13}).
\end{proof}

\section{Orthogonal Ricci Curvature}
We prove part (3) of Theorem \ref{thm algebra R} in this section. 
For convenience, we restate it in an equivalent way with the only difference being shifting the dimension from $m$ to $m+1$.

\begin{proposition}\label{prop Ric perp m+1}
    Let $R$ be a K\"ahler algebraic curvature operator on a complex Euclidean vector space $(V,g,J)$ of complex dimension $(m+1)\geq 2$. 
    Let \begin{equation*}
        \tilde{\b}_m:=\b_{m+1}=\frac{m(m+2)(3m+5)}{2(m+1)}.
    \end{equation*}
If $R$ has $\tilde{\b}_m$-nonnegative (respectively, $\tilde{\b}_m$-positive, $\tilde{\b}_m$-nonpositive, $\tilde{\b}_m$-negative) curvature operator of the second kind, then $R$ has nonnegative (respectively, positive, nonpositive, negative) orthogonal Ricci curvature. 
\end{proposition}

\begin{proof}
The key idea is to use $\CP^m \times \CP^1$ as a model space. The eigenvalues and their associated eigenvectors of the curvature operator of the second kind on $\CP^m \times \CP^1$ are determined in \cite{Li22product}. 

We construct an orthonormal basis of $S^2_0(V)$.  
Let $\{e_0, \cdots, e_m, Je_0, \cdots, Je_m\}$ be an orthonormal basis of $V$. 
Let 
\begin{align*}
    V_0 &=\spn\{e_0,Je_0\}, \\
    V_1 &=\spn\{e_1, \cdots, e_m, Je_1, \cdots, Je_m \}.
\end{align*} 
Then $V=V_0\oplus V_1$.

Notice that $h \in S^2_0(V_i)$ can be viewed as an element in $S^2_0(V)$ via 
\begin{equation*}
    h(X, Y)=h(\pi_i(X),\pi_i(Y))
\end{equation*}
where $\pi_i:V \to V_i$ is the projection map for $i=1,2$. 
Thus, $S^2_0(V_i)$ can be identified with a subspace of $S^2_0(V)$ for $i=1,2$.

Let $\mathcal{E}^+\cup \mathcal{E}^-$ be the orthonormal basis of $S^2_0(V_1)$ constructed in Lemma \ref{lemma basis +} and Lemma \ref{lemma basis -}. An orthonormal basis of $S^2_0(V_0)$ is given by 
\begin{eqnarray*}
    \tau_1 &=& \frac{1}{2\sqrt{2}} \left(e_0 \odot e_0 -Je_0 \odot Je_0 \right), \\
    \tau_2 &=& \frac{1}{\sqrt{2}} \left( e_0 \odot Je_0 \right)
\end{eqnarray*}
Next, we define the following traceless symmetric two-tensors on $V$: 
\begin{eqnarray*}
    h_i &=& \frac{1}{\sqrt{2}} e_0 \odot e_i \text{ for } 1\leq i\leq m, \\
    h_{m+i} &=& \frac{1}{\sqrt{2}} e_0 \odot J e_i \text{ for } 1\leq i\leq m,  \\
    h_{2m+i} &=& \frac{1}{\sqrt{2}} Je_0 \odot  e_i \text{ for } 1\leq i\leq m,  \\
    h_{3m+i} &=& \frac{1}{\sqrt{2}} Je_0 \odot  J e_i \text{ for } 1\leq i\leq m. 
\end{eqnarray*}
We also define 
\begin{equation*}
    \zeta= \frac{1}{\sqrt{8m(m+1)}} \left(m(e_0\odot e_0 +Je_0 \odot Je_0) -\sum_{i=1}^m (e_i\odot e_i +Je_i \odot Je_i)\right).
\end{equation*}
Then it is a straightforward computation to verify that 
$$\mathcal{E}^+\cup \mathcal{E}^- \cup \{\tau_i\}_{i=1}^2 \cup \{h_i\}_{i=1}^{4m} \cup \{\zeta\} $$ forms an orthonormal basis of $S^2_0(V)$. 
This corresponds to the orthogonal decomposition
\begin{equation*}
    S^2_0(V)=S^2_0(V_1) \oplus S^2_0(V_0) \oplus \spn\{u \odot v: u \in V_0, v\in V_1 \} \oplus \R \zeta.
\end{equation*}

The next step is to calculate the diagonal elements of the matrix representation of $\mathring{R}$ with respect to this basis. 
By Lemma \ref{lemma ij}, we have
\begin{equation}\label{eq R gamma}
    \mathring{R}(\tau_1,\tau_1)=\mathring{R}(\tau_2,\tau_2)=R(e_0,Je_0,e_0,Je_0). 
\end{equation}
Using Lemma \ref{lemma ij} again, we obtain 
\begin{eqnarray*}
     \sum_{i=1}^{4m} \mathring{R}(h_i,h_i) 
    &=& \sum_{i=1}^{m}R(e_0,e_i,e_0,e_i) +\sum_{i=1}^{m}R(e_0,Je_i,e_0,Je_i)  \\
   && +  \sum_{i=1}^{m}R(Je_0,e_i,Je_0,e_i) +\sum_{i=1}^{m}R(Je_0,Je_i,Je_0,Je_i)  \\
   &=& 2 \sum_{i=1}^{m} \left( R(e_0,e_i,e_0,e_i) + R(e_0,Je_i,e_0,Je_i) \right) \\
   &=& 2 \sum_{i=1}^{m}R(e_0,Je_0,e_i,Je_i) \\
   &=& 2\Ric^\perp (e_0,e_0).
\end{eqnarray*}
We calculate using \eqref{eq R iiJiJi} that 
\begin{eqnarray*}
    \mathring{R}(\zeta,\zeta) &=&  \frac{m^2}{8m(m+1)} \mathring{R}(e_0\odot e_0 +Je_0 \odot Je_0, e_0\odot e_0 +Je_0 \odot Je_0) \\
    && - \frac{2m}{8m(m+1)} \sum_{i=1}^m\mathring{R}(e_0\odot e_0 +Je_0 \odot Je_0, e_i\odot e_i +Je_i \odot Je_i) \\
    && +\frac{1}{8m(m+1)} \sum_{i,j=1}^m \mathring{R}(e_i\odot e_i +Je_i \odot Je_i, e_j\odot e_j +Je_j \odot Je_j)\\
     &=& -\frac{m}{(m+1)}R(e_0,Je_0,e_0,Je_0) +\frac{2}{m+1} \Ric^\perp (e_0,e_0) \\
    && -\frac{1}{m(m+1)}\sum_{i,j=1}^m R(e_i,Je_i,e_j,Je_j).
\end{eqnarray*}

Let $A$ be the collection of the values $\mathring{R}(\tau_i,\tau_i)$ for $i=1,2$, $\mathring{R}(\theta_i,\theta_i)$ for $1\leq i \leq 2m$, $\mathring{R}(\vp^{+}_{ij},\vp^{+}_{ij})$ and $\mathring{R}(\psi^{+}_{ij},\psi^{+}_{ij})$ for $1\leq i < j \leq m$.
By Lemma \ref{lemma R basis +} and \eqref{eq R gamma}, we know that $A$ contains two copies of $R(e_i,Je_i,e_i,Je_i)$ for each $0\leq i \leq m$ and two copies of $2R(e_i,Je_i,e_j,Je_j)$ for each $1\leq i < j \leq m$. 
Therefore, $\bar{a}$, the average of all values in $A$, is given by 
\begin{eqnarray*}
    \bar{a} = \frac{2}{m^2+m+2} \left(  \sum_{i,j=1}^m R(e_i,Je_i,e_j,Je_j) +R(e_0,Je_0,e_0,Je_0)\right) .
\end{eqnarray*}

Since $R$ has $\tilde{\b}_m$-nonnegative curvature operator of the second kind with 
\begin{equation*}
        \tilde{\b}_m=(m^2-1)+4m+1+\frac{m(m^2+m+2)}{2(m+1)},
\end{equation*}
we obtain that
\begin{eqnarray}\label{eq R nonnegative Ricci perp}
 0 &\leq & \sum_{1\leq i <j \leq m} \left( \mathring{R}(\vp^-_{ij}, \vp^-_{ij}) + \mathring{R}(\psi^-_{ij}, \psi^-_{ij}) \right) +\sum_{k=1}^{m-1} \mathring{R}(\eta_k,\eta_k)  \\ 
 && +\sum_{i=1}^{4m} \mathring{R}(h_i,h_i) + \mathring{R}(\zeta,\zeta) + f\left(A, \frac{m(m^2+m+2)}{2(m+1)}\right), \nonumber
\end{eqnarray}
where $f(A,x)$ is the function defined in Lemma \ref{lemma average}.
By Lemma \ref{lemma average}, we have 
\begin{eqnarray}\label{eq f estimate}
   && f\left(A, \frac{m(m^2+m+2)}{2(m+1)}\right) 
   \leq  \frac{m(m^2+m+2)}{2(m+1)} \bar{a} \\
    &=& \frac{m}{m+1}\left(  \sum_{i,j=1}^m R(e_i,Je_i,e_j,Je_j) -R(e_0,Je_0,e_0,Je_0)\right). \nonumber
\end{eqnarray}
By \eqref{eq sum of negative eigenvaleus} and \eqref{eq scalar curvature}, we have  
\begin{eqnarray}\label{eq sum of negative eigenvalues new}
 && \sum_{1\leq i <j \leq m} \left( \mathring{R}(\vp^-_{ij}, \vp^-_{ij}) + \mathring{R}(\psi^-_{ij}, \psi^-_{ij}) \right) +\sum_{k=1}^{m-1} \mathring{R}(\eta_k,\eta_k) \\
 &=&   -\frac{m-1}{m}\sum_{i,j=1}^m R(e_i,Je_i,e_j,Je_j). \nonumber
\end{eqnarray}
Substituting \eqref{eq sum of negative eigenvalues new}, the expressions for $\sum_{i=1}^{4m} \mathring{R}(h_i,h_i)$ and $\mathring{R}(\zeta,\zeta)$, and \eqref{eq f estimate} into \eqref{eq R nonnegative Ricci perp} yields
\begin{eqnarray*}
 0 &\leq &  -\frac{m-1}{m}\sum_{i,j=1}^m R(e_i,Je_i,e_j,Je_j)+ 2\Ric^\perp (e_0,e_0) \\
&&-\frac{m}{m+1}R(e_0,Je_0,e_0,Je_0) +\frac{2}{m+1}  \Ric^\perp(e_0,e_0) \\
&& -\frac{1}{m(m+1)}\sum_{i,j=1}^m R(e_i,Je_i,e_j,Je_j)\\
&& +\frac{m}{m+1}\left(  \sum_{i,j=1}^m R(e_i,Je_i,e_j,Je_j) +R(e_0,Je_0,e_0,Je_0)\right) \\
&=& \frac{2(m+2)}{m+1}\Ric^\perp(e_0,e_0).
\end{eqnarray*}
Since the orthonormal frame $\{e_0, \cdots, e_m, Je_0, \cdots, Je_m\}$ is arbitrary, we conclude that $\Ric^\perp \geq 0$. 
Other statements can be proved similarly. 
\end{proof}

\begin{proof}[Proof of Theorem \ref{thm Ric perp}]
By Proposition \ref{prop OB}, $M$ has positive orthogonal Ricci curvature. If $M$ is in addition closed, one can use \cite[Theorem 2.2]{Ni21} to conclude that $h^{p,0}=0$ for all $1\leq p\leq m$. In particular, $M$ is simply-connected and projective.
\end{proof}

By the classification of closed three-dimensional K\"ahler manifolds with positive orthogonal Ricci curvature in \cite{NWZ21}, we get the following corollary of Theorem \ref{thm Ric perp}. 
\begin{corollary}\label{corollary Ric perp 3D}
A closed three-dimensional K\"ahler manifold with $\frac{44}{3}$-positive curvature operator of the second kind is either biholomorphic to $\CP^3$ or biholomorphic to $\mathbb{Q}^3$, the smooth quadratic hypersurface in $\CP^4$. 
\end{corollary}

\section{Mixed Curvature}



Part (4) of Theorem \ref{thm algebra R} follows immediately from the following proposition. 
\begin{proposition}\label{prop mixed curvature}
    Let $(V,g,J)$ be a complex Euclidean vector space with complex dimension $(m+1)\geq 2$. Let $R$ be a K\"ahler algebraic curvature operator on $V$. 
    Let $$\tilde{\g}_m:=\g_{m+1}=\frac{3m^2+8m+4}{2}.$$
   If $R$ has $\tilde{\g}_m$-nonnegative (respectively, $\tilde{\g}_m$-positive, $\tilde{\g}_m$-nonpositive, $\tilde{\g}_m$-negative) curvature operator of the second kind, then the expression 
    \begin{equation*}
        2\Ric(X,X)-R(X,JX,X,JX)/|X|^2 
    \end{equation*}
    is nonnegative (respectively, positive, nonpositive, negative) for any nonzero vector $X$ in $V$. 
\end{proposition}

\begin{proof}[Proof of Proposition \ref{prop mixed curvature}]
The proof is similar to that of Proposition \ref{prop Ric perp m+1}, but with the different model space $\CP^m \times \C$.  The eigenvalues and their associated eigenvectors of the curvature operator of the second kind on $\CP^m \times \C$ are determined in \cite{Li22product}. 

Let $\{e_0, \cdots, e_m, Je_0, \cdots, Je_m\}$ be an orthonormal basis of $V$. 
Then $V=V_0\oplus V_1$, where $V_0 =\spn\{e_0,Je_0\}$ and $ V_1 =\spn\{e_1, \cdots, e_m, Je_1, \cdots, Je_m \}$.
Let $$\mathcal{E}^+\cup \mathcal{E}^- \cup \{\tau_i\}_{i=1}^2 \cup \{h_i\}_{i=1}^{4m} \cup \{\zeta\} $$ be the orthonormal basis of $S^2_0(V)$ constructed in the proof of Proposition \ref{prop Ric perp m+1}. 

Let $A$ be the collection of the values $\mathring{R}(\theta_i,\theta_i)$ for $1\leq i \leq 2m$, $\mathring{R}(\vp^{+}_{ij},\vp^{+}_{ij})$ and $\mathring{R}(\psi^{+}_{ij},\psi^{+}_{ij})$ for $1\leq i < j \leq m$.
By Lemma \ref{lemma R basis +}, we know that $A$ contains two copies of $R(e_i,Je_i,e_i,Je_i)$ for each $1 \leq i \leq m$ and two copies of $2R(e_i,Je_i,e_j,Je_j)$ for each $1\leq i < j \leq m$. 
Therefore, $\bar{a}$, the average of all values in $A$, is given by 
\begin{eqnarray*}
    \bar{a} = \frac{2}{m(m+1)} \left( \sum_{i,j=1}^m R(e_i,Je_i,e_j,Je_j) \right).
\end{eqnarray*}

Let $f(A,x)$ be the function defined in Lemma \ref{lemma average}. Then we have 
\begin{eqnarray}\label{eq 5.1}
    f\left(A, \frac{1}{2}m^2\right) \leq \frac{1}{2}m^2 \bar{a}=\frac{m}{m+1}\sum_{i,j=1}^m R(e_i,Je_i,e_j,Je_j).
\end{eqnarray}

Note that 
\begin{equation*}
    \tilde{\g}_m=(m^2-1)+4m+3 + \frac{1}{2}m^2.
\end{equation*}
Since $R$ has $\tilde{\g}_m$-nonnegative curvature operator of the second kind, 
we obtain that
\begin{eqnarray*}
 0 &\leq & \sum_{1\leq i <j \leq m} \left( \mathring{R}(\vp^-_{ij}, \vp^-_{ij}) + \mathring{R}(\psi^-_{ij}, \psi^-_{ij}) \right) +\sum_{k=1}^{m-1} \mathring{R}(\eta_k,\eta_k)  \\ 
 &&  +\sum_{i=1}^{4m} \mathring{R}(h_i,h_i)  + \mathring{R}(\zeta,\zeta) +  \mathring{R}(\tau_1,\tau_1)+  \mathring{R}(\tau_2,\tau_2) \nonumber \\
 && + f\left(A, \frac{1}{2}m^2\right). \nonumber
\end{eqnarray*}
Substituting \eqref{eq sum of negative eigenvalues new}, \eqref{eq 5.1}, \eqref{eq R gamma}, and the expressions for $\sum_{i=1}^{4m} \mathring{R}(h_i,h_i)$ and $\mathring{R} (\zeta,\zeta)$ into the above inequality produces 
\begin{eqnarray*}\label{eq R nonnegative Ricci perp mixed}
 0 &\leq &  -\frac{m-1}{m} \sum_{i,j=1}^m R(e_i,Je_i,e_j,Je_j)   + 2\Ric^\perp(e_0,e_0) \\
&&-\frac{m}{m+1}R(e_0,Je_0,e_0,Je_0) +\frac{2}{m+1} \Ric^\perp(e_0,e_0) \\
&& -\frac{1}{m(m+1)}\sum_{i,j=1}^m R(e_i,Je_i,e_j,Je_j) +2R(e_0,Je_0,e_0,Je_0)  \\
&& +\frac{m}{m+1}\sum_{i,j=1}^m R(e_i,Je_i,e_j,Je_j)\\
&=& \frac{m+2}{m+1}\left(2\Ric(e_0,e_0)-R(e_0,Je_0,e_0,Je_0) \right).
\end{eqnarray*}
Finally, the arbitrariness of the orthonormal frame $\{e_0, \cdots, e_m, Je_0, \cdots, Je_m\}$ allows us to conclude that 
$$2\Ric(X,X)-R(X,JX,X,JX)/|X|^2 \geq 0$$
for any nonzero vector $X \in V$. Other statements can be proved similarly. 

\end{proof}

\section*{Acknowledgments}
I am grateful to the anonymous referees for their valuable suggestions and comments. In particular, one referee pointed out the simple proof of \eqref{eq sum of negative eigenvaleus}, saving two pages of computations. 

\bibliographystyle{alpha}
\bibliography{ref}

\end{document}